\documentclass{article}
\usepackage[utf8]{inputenc}
\usepackage{fullpage}
\usepackage{amsmath}
\usepackage{amsfonts}
\usepackage{amssymb}
\usepackage{tikz-cd}
\usepackage{color,soul}
\usepackage{array}
\usepackage{zref-savepos}
\usepackage{amsthm}
\usepackage{multirow, bigstrut}
\usepackage[font={small,it}]{caption}
\usepackage{tikz}
\usepackage{amsmath}
\usepackage{arydshln}
\usetikzlibrary{arrows}

\usepackage{hyperref}
\hypersetup{
    colorlinks=true,
    linkcolor=blue,
    filecolor=magenta,  
    urlcolor=cyan,
}

\usetikzlibrary{calc}
\usepackage{comment}
\usepackage{enumerate}
\usepackage{tikz}
\usetikzlibrary{shapes.geometric}

\newtheorem{theorem}{Theorem}[section]
\newtheorem*{theorem*}{Theorem}
\newtheorem{lemma}[theorem]{Lemma}
\newtheorem{proposition}[theorem]{Proposition}
\newtheorem{corollary}[theorem]{Corollary}

\theoremstyle{definition}
\newtheorem{definition}[theorem]{Definition}
\newtheorem{remark}[theorem]{Remark}

\theoremstyle{plain}

\newcommand{\C}{\mathbb{C}}
\newcommand{\R}{\mathbb{R}}
\newcommand{\N}{\mathbb{N}}

\DeclareMathOperator{\Tr}{Tr}

\DeclareMathOperator{\diag}{diag}

\DeclareMathOperator{\Stab}{Stab}
\DeclareMathOperator{\Ad}{Ad}
\DeclareMathOperator{\ad}{ad}

\DeclareMathOperator{\hull}{hull}
\DeclareMathOperator{\spn}{span}

\DeclareMathOperator{\primal}{Prim}
\DeclareMathOperator{\dual}{Dual}
\DeclareMathOperator{\poly}{poly}

\begin{document}

\title{\bf On the Computability of Continuous Maximum Entropy Distributions: Adjoint Orbits of Lie Groups}
\author{Jonathan Leake \\ TU Berlin \and Nisheeth K. Vishnoi \\Yale University}

\maketitle

\newcommand{\eps}{\varepsilon}

\abstract{
Given a point $A$ in the convex hull of a given adjoint orbit $\mathcal{O}(F)$ of a compact Lie group $G$,   we give a polynomial time algorithm to compute the probability density supported on $\mathcal{O}(F)$ whose expectation is $A$ and that minimizes the Kullback-Leibler divergence to the $G$-invariant measure on $\mathcal{O}(F)$. 
This significantly extends the recent work of the authors \cite{leake2020computability} that presented such a result for  the manifold of rank $k$-projections which is a specific adjoint orbit of the unitary group $\mathrm{U}(n)$.
Our result relies on the ellipsoid method-based framework proposed in  \cite{leake2020computability}; however, to apply it to the general setting of compact Lie groups, we need   tools from Lie theory.
For instance, properties of the adjoint representation are used to find the defining equalities of the minimal affine space containing the convex hull of $\mathcal{O}(F)$, and to establish a bound on the optimal dual solution. 
Also, the Harish-Chandra integral formula  \cite{HarishChandra1957} is used to obtain an evaluation oracle for the dual objective function.
While the Harish-Chandra integral formula allows us to write certain integrals over the adjoint orbit of a Lie group as a sum of a small number of determinants, it  is only defined for elements of a chosen Cartan subalgebra of the Lie algebra $\mathfrak{g}$ of $G$.
We show how it can be applied to our setting with the help of  Kostant's convexity theorem \cite{Kostant1973}.
Further, the convex hull of an adjoint orbit is a type of {\em orbitope} \cite{Barvinok05convexgeometry,sanyal2011,biliotti2014}, and the orbitopes studied in this paper are known to be spectrahedral \cite{Kobert}.
Thus our main result can be viewed as extending the maximum entropy framework to a class of spectrahedra.
}

\newpage

\tableofcontents

\newpage

\section{Introduction}
A Lie group is a smooth manifold with a group structure where the group  multiplication and inverse are smooth maps.
Lie groups are widely used in many parts of modern mathematics and physics and include various matrix groups such as   linear groups, unitary groups, orthogonal groups, and symplectic groups that arise in various branches of science and engineering \cite{baker2003matrix,hall2003lie,gilmore2008lie,Edelman1999}.
Recently, Lie groups and algorithmic problems related to them have arisen in various works in theoretical computer science; see, e.g., \cite{christandl2014,MSSconvolutions,GargGOW16, GargGOW17, BurgisserFGOWW18,BurgisserFGOWW19,BLNW}.

Before we explain the problem studied in this paper, we recall some basic notions from Lie theory; see Section \ref{sec:preliminaries} for formal definitions. 
For a Lie group $G$, the tangent  space at its identity element is denoted by $\mathfrak{g}$ (this is a vector space); it inherits a multiplication-like operation (Lie bracket) from $G$ and is called the Lie algebra. 
The {\em adjoint action} ({\em adjoint representation} is also used) of $G$  is the differential of the conjugation action of the group on itself, evaluated at identity.
This can be thought of as a way of representing the elements of  $G$ as linear transformations of $\mathfrak{g}$. 
For instance, if $G$ is the unitary group  $\mathrm{U}(n)$, its Lie algebra $\mathfrak{u}(n)$  consists of $n \times n$ skew-Hermitian matrices and the adjoint action is: 
$$ \Ad_g:\mathfrak{u}(n) \to \mathfrak{u}(n) \ \ \ \ \ \Ad_g(X) \mapsto gXg^{-1}.$$

\noindent
Given an $F \in \mathfrak{g}$, the adjoint orbit of $F$ is defined as the image of $F$ under the adjoint action:
\[
    \mathcal{O}(F) := \bigcup_{g \in G} \{\Ad_g F\}.
\]
$ \mathcal{O}(F)$ is a nonconvex subset of the vector space $\mathfrak{g}$ and we let $\hull( \mathcal{O}(F))$ denote the convex hull of $ \mathcal{O}(F)$.

Associated to every compact Lie group $G$ is a canonical $G$-invariant probability measure called the \emph{Haar measure}, where $G$ acts by multiplication on the left or right.
Because of the invariance properties, we may also refer to this measure as the \emph{uniform measure}.
It follows from Lie theory that, for any given $F$, there is a canonical measure on $\mathcal{O}(F)$ which is the unique (up to scalar) $\Ad$-invariant measure on the orbit.
We denote this measure by $\mu_F$, and we often refer to it as the \emph{uniform measure} on $\mathcal{O}(F)$ due to its invariance properties.

\paragraph{The problem studied in this paper.} For a compact Lie group $G$, given an $F\in \mathfrak{g}$ and a point $A \in \hull( \mathcal{O}(F))$, we consider the problem of computing the probability density supported on $\mathcal{O}(F)$ whose expectation is $A$ and that minimizes the Kullback-Leibler divergence to the invariant measure $\mu_F$ on $\mathcal{O}(F)$. Since, by choice, the measure we seek is as close to
the distribution $\mu_F$ as possible, we refer to it as a maximum entropy distribution.

While the problem of finding the optimal distribution is convex (see Figure \ref{fig:primal_dual}), it is not immediately obvious how to succinctly represent the optimal density function.
For this, consider the dual: 
\begin{equation}\label{eq:1.dual}
 \inf_{Y \in \mathfrak{g}} \ \langle Y,A\rangle + \log \int_{X \in \mathcal{O}(F)} e^{-\langle Y, X \rangle } d\mu_F(X),
\end{equation} 
where $\langle \cdot, \cdot \rangle$  is an inner product on $\mathfrak{g}$ and $\mu_F$ is the given  measure.
If strong duality holds, it can be shown that the optimal distribution $\nu^\star$ to the entropy maximizing problem above can be described by the optimizer $Y^\star$ to the dual above:
$$
 \nu^\star (X) \propto e^{-\langle Y^\star,X \rangle}
$$
   for $X \in \mathcal{O}(F).$
As for computability of $\nu^\star$, $Y^\star$ lives in a small, convex, and finite-dimensional  domain and in principle could be found using the ellipsoid method.
However, bounding the running time of such an   optimization method to find $Y^\star$ reduces to 1) identifying the maximal affine space in which $\mathcal{O}(F)$ lives, 2)   bounding some norm of $Y^\star$ and, 3) coming up with efficient algorithms to compute  
\begin{equation}\label{eq:1.count}
\int_{X \in \mathcal{O}(F)} e^{-\langle Y, X \rangle } d\mu_F(X)
 \end{equation}
 for matrices $Y$ with at most that norm.

\paragraph{A special case.} A special case of the above problem for $\mathrm{U}(n)$ was studied recently \cite{leake2020computability}.
The unitary group $\mathrm{U}(n)$ acts on the real vector space of $n \times n$ Hermitian matrices (which, up to a multiplication by $i$ is its Lie algebra $\mathfrak{u}(n)$) by conjugation.
Thus,  this vector space is partitioned into orbits, with $X$ and $Y$ being in the same orbit if and only if they have the same eigenvalues.
Consider now the matrix $P_k := \diag(1,\ldots,1,0,\ldots,0)$ where $k$ denotes the number of $1$s that appear in the matrix.
Then the orbit $\mathcal{O}(P_k)$ is precisely the set of rank-$k$ projections.
The main result of  \cite{leake2020computability} is an ellipsoid method-based polynomial time algorithm to solve the dual optimization problem in Equation \eqref{eq:1.dual} for a given $A \in \hull \mathcal{O}(P_k)$; see Theorem  \ref{thm:general_algo}.
The algorithm,
roughly speaking, ran in time polynomial in $n$, $\log 1/\eps$, bit complexity of the input, and $\eta^{-1}$ where $\eta$ is a number such that a ball of radius $\eta$ centered at $A$ is contained in $ \hull \mathcal{O}(P_k).$

\renewcommand{\epsilon}{\varepsilon}

\paragraph{Our contribution.} The main result of this paper is an extension of the main result of \cite{leake2020computability} (Theorem \ref{thm:general_algo})  to the setting of adjoint orbits of compact Lie groups; see Theorem \ref{thm:lie_algo}.
Our algorithm uses the ellipsoid method and runs in time polynomial in $\dim(\mathfrak{g})$, $\eta^{-1}$, $\log(\epsilon^{-1})$, and the number of bits needed to represent $A$ and $F$.
Compared to \cite{leake2020computability}, the contributions of this paper are in synthesizing various Lie theoretic tools which are required to apply the ellipsoid method-framework.
In particular, we use the following to do this:
\begin{enumerate}
    \item \textbf{Harish-Chandra integral formula (Theorem \ref{thm:HC}).}
    In \cite{leake2020computability}, the Harish-Chandra-Itzykson-Zuber (HCIZ) integral formula \cite{HarishChandra1957,IZ1980,DuistermaatH1982,Vergne1996,Tao} was needed to compute integrals on orbits of the unitary group.
    The more general Harish-Chandra integral formula,   \cite{HarishChandra1957,mcswiggen2018harish}, discovered while working towards developing a theory of
Fourier analysis on semisimple Lie algebras,  gives a way  to convert an orbital integral such as that of Equation \eqref{eq:1.count} to a sum over the corresponding finite Weyl group.
However,  it is not clear how to compute this latter sum as it may be exponential.
While the HCIZ formula provides a way to represent the summation over the Weyl group as a determinant of a small matrix,  for other Lie groups such formulae are more difficult to obtain and can be found in a recent work of \cite{mcswiggen2018harish}.
We present them in Section \ref{sec:applications}.

    \item \textbf{Kostant convexity theorem (Theorem \ref{thm:kostant}).}
    One issue is that  the Harish-Chandra formula only takes elements from the Cartan subalgebra of $\mathfrak{g}$ as input, while the formula in Equation \eqref{eq:1.count} allows $Y$ to be any element of $\mathfrak{g}$.
    In the case of $P_k$, the Schur-Horn theorem was used to show that the optimal solution to dual optimization problem is a diagonal matrix. 
    Here we appeal to its far-reaching generalization called the Kostant convexity theorem to prove that the optimal solution to Equation \eqref{eq:1.dual} lies in our choice Cartan subalgebra.
    This allows us to efficiently compute the necessary integrals on an adjoint orbit (Equation \eqref{eq:1.count}), leading to a  counting oracle for the ellipsoid method.

    \item \textbf{Adjoint representation (Section \ref{sec:linear_equalities}).}
    To apply the ellipsoid method, one must be able to describe the minimal affine space in which $\mathcal{O}(F)$ lives (see Theorem \ref{thm:general_algo}).
    In the case of $P_k$ studied in \cite{leake2020computability}, this affine space is trivially described as the set of all trace-$k$ Hermitian matrices.
    In the case of adjoint orbits of compact Lie groups however, this description could be more complicated.
    That said, we are able to generally describe this minimal affine space by viewing the span of the adjoint orbit as a representation of the corresponding compact Lie group.
    Reducibility properties of this representation then yield defining linear equalities for the minimal affine space.

    \item \textbf{The Killing form for compact groups.}
    The dual optimization problem (Equation \eqref{eq:1.dual}) inherently relies upon a choice of inner product $\langle \cdot, \cdot \rangle$ for the space in which the maximum entropy measures lie.
    In \cite{leake2020computability}, the Frobenius inner product was chosen seemingly out of convenience.
    But in order to use the Harish-Chandra formula (Theorem \ref{thm:HC}), we need this inner product to be invariant with respect to the adjoint action on the Lie algebra.
    When the Lie group at hand is compact, all such invariant inner products are essentially derived from the Killing form.
    Beyond the Harish-Chandra formula, this invariance property is also crucial for our bounding box arguments and the use of the Kostant convexity theorem.
    
    \item \textbf{Bounding box for adjoint orbits (Section \ref{sec:bounding_box}).}
    In \cite{leake2020computability}, the notion of {\em balancedness} of a measure was introduced to give a nice criterion for existence of a bounding box for the solution to the dual optimization problem.
    Balancedness conceptually means that the measure is roughly evenly spread out.
    In \cite{leake2020computability}, unitary invariance allowed one to view the measures at hand as uniform, which implied balancedness.
    In this paper we can extend this argument by considering more general $\Ad$-invariant measures, which can be extracted from the Haar measure on the associated Lie group.

\end{enumerate}

\noindent
Finally, in Section \ref{sec:applications}, we show how our result can be applied to some  compact Lie groups such as unitary groups ($\mathrm{U}(n)$ and $\mathrm{SU}(n)$), orthogonal groups ($\mathrm{O}(n)$ and $\mathrm{SO}(n)$), and  symplectic groups ($\mathrm{USp}(n)$, the underlying symmetry group for Hamiltonian dynamics). 
For those groups we determine a Cartan subalgebra and an appropriate orthonormal basis with respect to the Killing form.
This allows us to use the determinantal formulas of \cite{mcswiggen2018harish} along with  the Harish-Chandra integral formula to be able to efficiently compute the exponential integral of Equation \ref{eq:1.count}.

While the main result shows how to compute an approximation $Y$ of $Y^\star$ efficiently,   it is an open problem to come up with an efficient  algorithm to sample from the distribution $\nu^\star(X) \propto e^{-\langle Y,X\rangle}$ given $F$ and $Y.$

\paragraph{Connection to orbitopes.} Convex bodies of the form $\hull(\mathcal{O}(F))$ are examples of  {\em orbitopes},  objects introduced in \cite{Barvinok05convexgeometry,sanyal2011}. 
In particular, the orbitopes that arise as adjoint orbits of compact Lie groups fit into the context of {\em coadjoint orbitopes} which have been studied in \cite{biliotti2014}.
Well-studied examples of adjoint orbitopes include:
\begin{itemize}
\item {\em Hermitian Schur-Horn orbitope.}  These arises by letting $G=\mathrm{U}(n)$ and letting $F$ be a diagonal matrix of given eigenvalues $(\lambda_1,\ldots,\lambda_n)$.
\item {\em Skew-symmetric Schur-Horn orbitope.} This arises by letting $G$ be $\mathrm{O}(n)$ or $\mathrm{SO}(n)$ and letting $F$ be the real block diagonal matrix of $2 \times 2$ matrices given by
\[
  \begin{bmatrix}
    0 & \lambda_1 \\
    -\lambda_1 & 0
  \end{bmatrix},
  \begin{bmatrix}
    0 & \lambda_2 \\
    -\lambda_2 & 0
  \end{bmatrix},
  \ldots,
  \begin{bmatrix}
    0 & \lambda_k \\
    -\lambda_k & 0
  \end{bmatrix},
\]
where $k = \lfloor \frac{n}{2} \rfloor$, and with an extra row and column of 0s if $n$ is odd.
\item {\em Density matrices in quantum mechanics.}  These arises by letting $G=\mathrm{U}(n)$ and letting $F$ be the diagonal matrix with the $(1,1)$ entry equal to $1$ and rest $0$.
\end{itemize}
\noindent
Interestingly, in a sequence of results \cite{Parrilo,BiliottiPolar,Kobert,Kobertphdthesis,kobert2020spectrahedral}, it has been established that adjoint orbitopes of compact Lie groups are spectrahedral, making it possible to optimize convex functions over them.
The results of this paper then imply the maximum entropy framework can be extended to a particular class of spectrahedra.

In summary, this paper shows how various results from Lie theory come together to allow for the computability of maximum entropy distributions on compact adjoint orbits. In doing so, our work  contributes to the algorithmic theory of coadjoint orbitopes \cite{Barvinok05convexgeometry,sanyal2011,Parrilo,Kobert} and, more generally, to recent work on algorithms for  optimization problems on Lie groups \cite{GargGOW16, GargGOW17, BurgisserFGOWW18,BurgisserFGOWW19,BLNW}.

\section{Preliminaries} \label{sec:preliminaries}

We let $\C,\R,\R_{\geq 0},\N$ denote the sets of complex numbers, real numbers, non-negative real numbers, and positive integers respectively.
Throughout we use standard vector and matrix notation.
We let $\R^{n \times m}$ and $\C^{n \times m}$ denote the vector spaces of $n \times m$ matrices with real and complex coefficients, respectively.
For $A \in \C^{n \times n}$, we let $\Tr(A)$ denote the trace of $A$.
We will often use $e_j$ to denote the standard basis vector and $E_{j,k}$ to denote the matrix with a one in the $j$th row and $k$th column and zeros elsewhere.

\subsection{Lie theory preliminaries}

We now discuss some basic notions related to Lie groups.
Lie theory in general can be introduced in a number of ways, and here we only discuss notions pertinent to the problems considered in this paper.
Further, we do not go into very much detail, instead pointing the interested reader to any standard Lie theory reference (such as \cite{knapp2013} or \cite{hall2003lie}).

\paragraph{Lie groups and Lie algebras.}

A \emph{(real) Lie group} $G$ is a real $d$-dimensional smooth manifold which is also a group with smooth group operations.
Given a Lie subgroup $H$ of $G$, the \emph{quotient} $G/H$ is a smooth manifold (not necessarily a group) consisting of the left cosets of $H$ in $G$.
As a note, we will almost exclusively work with compact connected Lie groups.

The real $d$-dimensional tangent space of $G$ at the identity is called the \emph{(real) Lie algebra} $\mathfrak{g}$ associated to $G$.
That is, a Lie algebra is a real vector space of the same dimension as the corresponding Lie group.
A Lie algebra is also equipped with a certain bilinear operation called the Lie bracket, and we will discuss this via the $\ad$ action below.

It should be noted that our Lie groups and Lie algebras will often lie in spaces of matrices with complex entries.
Even still, we will consider them to be \emph{real} Lie groups and Lie algebras.
Often (as in the case of the unitary group) this assumption is required, as the Lie groups of interest do not necessarily admit a complex structure.
In any case, this will not affect our computations.

We now give a few examples of Lie groups and Lie algebras which will be of particular importance for us.
First, the groups $GL_n(\C)$ and $GL_n(\R)$ are the Lie groups of complex and real invertible $n \times n$ matrices respectively, where the group action is given by matrix multiplication.
The corresponding Lie algebras are given by $\mathfrak{gl}_n(\C) = \C^{n \times n}$ and $\mathfrak{gl}_n(\R) = \R^{n \times n}$, the vector spaces of all $n \times n$ complex and real matrices respectively.
More generally, we can also consider the Lie group $GL(V)$ and corresponding Lie algebra $\mathfrak{gl}(V)$ for a given vector space $V$.

Any closed subgroup of $GL_n$ is also a Lie group called a \emph{matrix (sub)group}.
(All explicit examples we consider in this paper are matrix groups.)
For example the groups $SL_n(\C) \subset GL_n(\C)$ and $SL_n(\R) \subset GL_n(\R)$ are the matrix subgroups of $n \times n$ matrices with determinant equal to 1.
The Lie algebras corresponding to $SL_n(\C)$ and $SL_n(\R)$, denoted $\mathfrak{sl}_n(\C) \subset \mathfrak{gl}_n(\C)$ and $\mathfrak{sl}_n(\R) \subset \mathfrak{gl}_n(\R)$, are the vector subspaces of traceless matrices.
Finally the groups $\mathrm{U}(n) \subset GL_n(\C)$ and $\mathrm{SU}(n) \subset SL_n(\C)$ are the matrix subgroups of $n \times n$ unitary matrices.
The Lie algebras corresponding to $\mathrm{U}(n)$ and $\mathrm{SU}(n)$, denoted $\mathfrak{u}(n) \subset \mathfrak{gl}_n(\C)$ and $\mathfrak{su}(n) \subset \mathfrak{sl}_n(\C)$, are the vector subspaces of skew-Hermitian matrices.

\paragraph{The adjoint action.}

The Lie group $G$ acts on $\mathfrak{g}$ in a canonical way called the \emph{adjoint action}.
This action, which we will denote by $\Ad_g X$ for $g \in G$ and $X \in \mathfrak{g}$, is defined in general as the differential at the identity of the action of $G$ on itself by conjugation.
The adjoint representation is in particular a Lie group homomorphism
\[
    \Ad: G \to GL(\mathfrak{g}).
\]
Since $\Ad$ is a map of smooth manifolds, one can further consider the differential of $\Ad$ at the identity.
Since $\mathfrak{g}$ can be thought of as the tangent space of $G$ at the identity, from this one obtains a Lie algebra homomorphism
\[
    \ad: \mathfrak{g} \to \mathfrak{gl}(\mathfrak{g}),
\]
given by $\ad_Y(X)$ for $X,Y \in \mathfrak{g}$.
The $\ad$ action is precisely the Lie bracket of $\mathfrak{g}$, and so some authors use the notation $[X, Y] := \ad_X(Y)$.
We will not use this notation here.

In the case that $G$ is a matrix group, we have that $\mathfrak{g}$ is a vector space of matrices and the actions can be given concretely as follows.

\begin{definition}
    Let $G$ be a matrix Lie group; i.e. $G \subseteq GL_n(\C)$ with $G$ closed and $\mathfrak{g} \subseteq \mathfrak{gl}_n(\C)$. Then the adjoint action of $G$ on its Lie algebra $\mathfrak{g}$ is given as
    \[
        \Ad_g X = g X g^{-1}
    \]
    for $g \in G$ and $X \in \mathfrak{g}$. Further, the adjoint action of $\mathfrak{g}$ on itself is given as
    \[
        \ad_X(Y) = XY - YX
    \]
    where $X,Y \in \mathfrak{g}$.
\end{definition}

\noindent
Finally, the \emph{orbit} of some given $F \in \mathfrak{g}$ with respect to the adjoint action is defined as usual as the set
\[
    \mathcal{O}(F) := \bigcup_{g \in G} \{\Ad_g F\},
\]
and the \emph{stabilizer} of $F$ is defined as usual as the subgroup
\[
    \Stab(F) := \{g \in G ~:~ \Ad_g F = F\}.
\]

\begin{remark}
    In this paper we consider adjoint orbits of compact Lie groups, but in the literature \emph{coadjoint orbits} are often considered.
    Coadjoint orbits are orbits of the dual action on $\mathfrak{g}^*$.
    In the case we consider, we have a nondegenerate $\Ad$-invariant bilinear form on $\mathfrak{g}$, and this implies adjoint and coadjoint orbits are equivalent.
    This is also discussed in Section \ref{sec:convexity}.
\end{remark}

\paragraph{The exponential map.}

A key fact which relates $\Ad$ to $\ad$ is that the \emph{exponential map} commutes with these actions.
We describe this in the following result, noting that the exponential map is a generalization of the usual matrix exponential.

\begin{proposition}[\cite{knapp2013}, Proposition 1.93 and Corollary 4.48] \label{prop:Lie_exp}
    If $G$ is a Lie group with Lie algebra $\mathfrak{g}$, then there is a map $\exp: \mathfrak{g} \to G$ such that
    \[
        \Ad_{\exp(X)} = \exp(\ad_X),
    \]
  where  $X \in \mathfrak{g}$.
    If $G$ is compact and connected then $\exp$ is surjective, and if $G$ is a matrix group then $\exp(X)$ is the usual matrix exponential.
\end{proposition}

\noindent
The exponential map also gives another connection between the actions $\Ad$ and $\ad$, which follows from the previous result.

\begin{corollary} \label{cor:exp_deriv}
    Given a Lie group $G$ with Lie algebra $\mathfrak{g}$, we have
    \[
        \ad_X(Y) = \left.\frac{d}{dt} \Ad_{\exp(tX)}(Y)\right|_{t=0}
    \]
    for any $X,Y \in \mathfrak{g}$.
\end{corollary}

\paragraph{Simple, semisimple, and reductive Lie algebras.}

Given a Lie algebra $\mathfrak{g}$, an element $X$ is in the \emph{center} of $\mathfrak{g}$ if for any $Y \in \mathfrak{g}$ we have $\ad_X(Y) = 0$.
(Note that $0$ is always in the center of $\mathfrak{g}$.)
A Lie algebra $\mathfrak{h} \subseteq \mathfrak{g}$ is a \emph{subalgebra} if it is closed under the action of $\ad$.
A Lie algebra $\mathfrak{g}$ is \emph{Abelian} if $\ad_X(Y) = 0$ for all $X,Y \in \mathfrak{g}$.
A Lie algebra $\mathfrak{g}$ is \emph{simple} if it has trivial center and no non-trivial subalgebras.
A Lie algebra $\mathfrak{g}$ is \emph{semisimple} if it can be decomposed as a direct sum of simple Lie algebras.
Finally, a Lie algebra $\mathfrak{g}$ is \emph{reductive} if it can be decomposed as a direct sum of its center and a number of simple Lie algebras; that is, if we can write
\[
    \mathfrak{g} = \mathfrak{z} \oplus \mathfrak{g}_1 \oplus \cdots \oplus \mathfrak{g}_n
\]
where $\mathfrak{z}$ is Abelian and $\mathfrak{g}_i$ is simple for all $i \in [n]$.
We refer to the direct sum of the simple components of a reductive Lie algebra as its \emph{semisimple part}.

The one result we state here will be used throughout this paper, often without mentioning it.

\begin{proposition}[\cite{knapp2013}, Corollary 4.25]
    If $G$ is a compact Lie group, then its Lie algebra $\mathfrak{g}$ is reductive.
\end{proposition}

\paragraph{Cartan subalgebras.}

Let $G$ be a compact connected Lie group, and let $\mathfrak{g}$ be its Lie algebra.
For our purposes, a \emph{Cartan subalgebra} $\mathfrak{h}$ of a Lie algebra $\mathfrak{g}$ is a generalization of the set of diagonal matrices in $\mathfrak{gl}_n(\C)$ or $\mathfrak{u}(n)$.
Cartan subalgebras are typically defined for complex Lie algebras, but in this paper we will identify them with \emph{maximal Abelian subalgebras} of $\mathfrak{g}$.
This identification is warranted, as some authors even use the term Cartan subalgebra to refer to maximal Abelian subalgebras in the compact connected case; see \cite{knapp2013} Section 4.5 (specifically the discussion on pages 200-201).
We discuss maximal Abelian (Cartan) subalgebras further in Section \ref{sec:convexity}.

\paragraph{The Killing form.}

There is a real bilinear form on $\mathfrak{g}$ called the \emph{Killing form} which is invariant with respect to the adjoint action.
Denoting this form by $B(\cdot, \cdot)$, this means
\[
    B(\Ad_g X, \Ad_g Y) = B(X, Y) \qquad \text{and} \qquad B(\ad_Z(X), Y) + B(X, \ad_Z(Y)) = 0
\]
for all $g \in G$ and $X,Y,Z \in \mathfrak{g}$.
When the Lie group $G$ is compact, we can say more.

\begin{proposition} \label{prop:Killing_orthogonal}
    Let $B$ be the Killing form on a given Lie algebra $\mathfrak{g}$ associated to a compact Lie group.
    If $\mathfrak{g}$ is simple, then $B$ is negative definite.
    If $\mathfrak{g}$ is reductive, then $B$ is a direct sum of the respective Killing forms on its simple components and the trivial form ($B = 0$) on its center.
    In particular, $B$ is negative semidefinite.
\end{proposition}

\paragraph{Invariant measures.}

Associated to every compact Lie group $G$ is a canonical $G$-invariant probability measure called the \emph{Haar measure}, where $G$ acts by multiplication on the left or right.
Because of the invariance properties, we may also refer to this measure as the \emph{uniform measure}.
We will often be integrating against this measure, and we will use the notation $\int_G f(g) dg$ to denote this.

Given an $F \in \mathfrak{g}$, we also consider measures on the orbit $\mathcal{O}(F)$ of the adjoint action of $G$.
For any given $F$, there is a canonical measure on $\mathcal{O}(F)$ which is the unique (up to scalar) $\Ad$-invariant measure on the orbit.
This is not a trivial fact, but is given in the following.

\begin{proposition}[see \cite{audin2012}, II.3.c and II.1.b]
    Let $G$ be a compact Lie group with Lie algebra $\mathfrak{g}$ and $\Ad$-invariant inner product $\langle \cdot, \cdot \rangle$. For any $F \in \mathfrak{g}$, there exists an $\Ad$-invariant measure $\mu_F$ supported on $\mathcal{O}(F)$.
\end{proposition}

\noindent
The measure $\mu_F$ is also unique (up to positive scalar), and we often refer to it as the \emph{uniform measure} on $\mathcal{O}(F)$ due to its invariance properties.
In Section \ref{sec:HC}, we discuss the connection between uniform measures on adjoint orbits and the Haar measure on $G$.

\section{Our Framework and Results}

\subsection{The Maximum Entropy Framework}

In this section we discuss the maximum entropy convex program for adjoint orbits.
Fix a compact connected real Lie group $G$ with corresponding Lie algebra $\mathfrak{g}$ equipped with an $\Ad$-invariant inner product $\langle \cdot, \cdot \rangle$. 
Further fix any $F \in \mathfrak{g}$, and let $\mathcal{O}(F) \subset \mathfrak{g}$ denote the corresponding adjoint orbit.
Let $\mu_F$ be the $\Ad$-invariant measure on $\mathcal{O}(F)$, and let $\mathcal{L}(X)=b$ denote the minimal affine space containing $\hull(\mathcal{O}(F))$.
Given $A \in \hull(\mathcal{O}(F))$, our goal is to find a density function $\nu$ with expectation $A$ that minimizes the KL-divergence with respect to $\mu_F$.

\begin{figure*}
    \centering
    \[
    \begin{array}[t]{c|c}
        \quad\quad
        \bf Primal
        \quad\quad\quad
        &
        ~
        \bf Dual
        \\

        \quad
        \begin{aligned}[t]
        &\sup_\nu \, -\int_{\mathcal{O}(F)} \nu(X) \log\left(\nu(X)\right) d\mu_F(X) \\
            &\textrm{subject to:} \\
                &\qquad \nu: \mathcal{O}(F) \to \R_{\geq 0}, \text{ $\mu_F$-measurable}, \\
                &\qquad \int_{\mathcal{O}(F)} X \nu(X) d\mu_F(X) = A, \\
                &\qquad \int_{\mathcal{O}(F)} \nu(X) d\mu_F(X) = 1.
        \end{aligned}
        \quad

        &

        ~
        \begin{aligned}[t]
        &\inf_{Y \in \mathfrak{g}}\, f_A(Y):=\langle Y, A \rangle + \log \int_{\mathcal{O}(F)} e^{-\langle Y, X \rangle} d\mu_F(X) \\
                &\textrm{subject to:} \\
                &\qquad \mathcal{L}(Y)=0.
        \end{aligned}
    \end{array}
    \]

    \caption{Primal and dual maximum entropy convex programs for $A$ in the interior of $\hull(\mathcal{O}(F))$.}
    \label{fig:primal_dual}
\end{figure*}

We use the shorthand $\primal_F(A)$ to refer to the primal convex optimization problem, and we use $\dual_F(A)$ to refer to the dual convex optimization problem.
Drawing from the intuition that $\mu_F$ is uniform over $\mathcal{O}(F)$, and hence in some sense maximizes entropy, we say the KL-divergence minimizing measure is entropy maximizing. 
The fact that the entropy integral (without the minus sign) is convex as a function of the density $\nu$ follows from the fact that this integral is precisely the KL divergence between the probability distribution corresponding to $\nu$ and the distribution $\mu_F$.
Convexity of the KL divergence for probability distributions is then a well-known fact.

\subsection{Duality and Strong Duality}

Efficiently solving the primal convex program directly is a priori impossible as the support of $\nu$ is infinite.
To find a succinct representation for the optimal $\nu^\star$, we turn to the dual program (see \cite{leake2020computability}, Appendix A for a derivation) which
gives us a nice representation of the maximum entropy density function $\nu^\star$.
By strong duality (see Theorem \ref{thm:strongduality} below), the maximum entropy density function $\nu^\star$ takes on a nice form:
\[
    \nu^\star(X) \propto e^{-\langle Y^\star, X \rangle}.
\]
Issues arising from non-uniqueness can be handled by restricting to the minimal affine subspace in which $\hull(\mathcal{O}(F))$ lives.
However, as $A$ tends to the boundary of $\hull(\mathcal{O}(F))$, $Y^\star$ can be seen to tend to infinity as the support of the measure $\nu^\star$ tends to lower dimensions.

That said, we now state the strong duality result.
This follows from Theorem A.4 of \cite{leake2020computability}.

\begin{theorem}[\textbf{Strong duality}]\label{thm:strongduality}
    Let $G$ be a compact connected real Lie group, and let $\mathfrak{g}$ be the corresponding Lie algebra equipped with $\Ad$-invariant inner product $\langle \cdot, \cdot \rangle$.
    Fix any $F \in \mathfrak{g}$, and let $\mu_F$ be the $\Ad$-invariant measure on $\mathcal{O}(F)$.
    For any $A$ in the relative interior of $\hull(\mathcal{O}(F))$,
    the optimal values of the primal and dual objective functions coincide, and the corresponding maximum entropy distribution has density function of the following form for some $Y^\star$:
    \[
        \nu^\star(X) \propto e^{-\langle Y^\star, X \rangle}.
    \]
\end{theorem}

\noindent
With strong duality in hand, we may now focus on solving the dual optimization program in order to compute maximum entropy distributions.

\subsection{Main Result}

In this section we state our main result, an algorithm for computing the maximum entropy density function.
Let $G$ be a compact connected real Lie group with corresponding Lie algebra $\mathfrak{g}$ equipped with an $\Ad$-invariant inner product $\langle \cdot, \cdot \rangle$.
Since $G$ is compact and hence $\mathfrak{g}$ is reductive, we consider the decomposition of $\mathfrak{g}$ into its center and simple components given by $\mathfrak{g} = \mathfrak{z} \oplus \mathfrak{g}_1 \oplus \cdots \oplus \mathfrak{g}_n$.
We then choose bases
\[
    \mathfrak{z} = \spn\{Z_1,\ldots,Z_{d_c}\} \qquad \text{and} \qquad \mathfrak{g}_i = \spn\{H^i_1, \ldots, H^i_{m_i}, X^i_1, \ldots, X^i_{d_i-m_i}\},
\]
where $H^i_1, \ldots, H^i_{m_i}$ span a Cartan subalgebra $\mathfrak{h}_i \subset \mathfrak{g}_i$ for all $i \in [n]$, and such that on $\mathfrak{g}_i$ the bases are orthonormal with respect to $\langle \cdot, \cdot \rangle$ for all $i \in [n]$.
We further denote $\mathfrak{h} := \mathfrak{h}_1 \oplus \cdots \oplus \mathfrak{h}_n$.

In our applications discussed in Section \ref{sec:applications}, the bases and inner product will be very explicit.
However, note that Proposition \ref{prop:Killing_orthogonal} and Corollary \ref{cor:extend_Killing} imply choosing such an orthonormal basis is always possible by using the Gram-Schmidt procedure with the inner product formed by adding any inner product on $\mathfrak{z}$ to $-B$.

With this setup in hand, we can now state our main result.
Recall that the $\eta$-interior of a set $\mathcal{K}$ is the set of all points in the relative interior of which are at least $\eta$ away from the boundary of $\mathcal{K}$.

\begin{theorem}[\textbf{General algorithm}] \label{thm:lie_algo}
    Fix a compact connected real Lie group $G$ and corresponding reductive Lie algebra $\mathfrak{g} = \mathfrak{z} \oplus \mathfrak{g}_1 \oplus \cdots \oplus \mathfrak{g}_n$.
    Let an $\Ad$-invariant inner product $\langle \cdot, \cdot \rangle$ and a basis of $\mathfrak{g}$ be given as in the above setup.

    There exists an algorithm that, given an $F \in \mathfrak{z} \oplus \mathfrak{h}$, an $A \in \mathfrak{z} \oplus \mathfrak{h}$ in the $\eta$-interior of $\hull(\mathcal{O}(F))$, an $\epsilon > 0$, and a strong counting/integration oracle for the exponential integral $\mathcal{E}_F(Y)$ restricted to $\mathfrak{h}$, returns $Y^\circ \in \mathfrak{h}$ such that
    \[
        f_A(Y^\circ) \leq f_A(Y^\star) + \epsilon,
    \]
    where $f_A$ is the objective function for the dual program $\dual_F(A)$, and $Y^\star$ is the optimum of the dual program.
    The running time of the algorithm is polynomial in $\dim(\mathfrak{g})$, $\eta^{-1}$, $\log(\epsilon^{-1})$, and the number of bits needed to represent $A$ and $F$.
\end{theorem}

\begin{remark}
    As discussed, compactness of the Lie group $G$ in the above theorem is crucial to our arguments because it implies one can use the Killing form to obtain an $\Ad$-invariant inner product on $\mathfrak{g}$.
    Connectedness is more subtle: a number of results we use require connectedness of the group, but not all.
    And, it seems plausible that one could extend these results to non-connected groups for our purposes.
    In particular, note that \cite{mcswiggen2018harish} contains effective formulas even for the disconnected case.
\end{remark}

\noindent
This result is related to Theorem 4.4 of \cite{leake2020computability} (see Theorem \ref{thm:general_algo} below), but its applications, given in Section \ref{sec:applications}, demonstrate the novelty of this paper with respect to \cite{leake2020computability}.
Specifically, the case of rank-$k$ projections studied in \cite{leake2020computability} is a special case of the above result for $G = \mathrm{U}(n)$ and $F = \diag(1,\ldots,1,0,\ldots,0)$.
In Section \ref{sec:applications}, we generalize this to different values of $F$ and to different compact connected Lie groups beyond $\mathrm{U}(n)$.

Before moving on to a full proof of this result, we first give a sketch.
The overarching idea is to use the ellipsoid method (Theorem \ref{thm:ellipsoid}) to approximate an optimal solution to the dual convex program $\dual_F(A)$.
To do this, we need to do a few things:

\begin{enumerate}
    \item Determine a small-enough ball containing $\mathcal{O}(F)$.
    \item Determine a bounding box for an optimal solution to $\dual_F(A)$.
    \item Determine a maximal set of affine equalities for $\hull(\mathcal{O}(F))$ and the corresponding bit complexity.
    \item Show that we can restrict our search space to $\mathfrak{h} \subset \mathfrak{g}$. (Note that the orbit $\mathcal{O}(F)$ is \emph{not} contained in $\mathfrak{h}$ or $\mathfrak{z} \oplus \mathfrak{h}$, even though $F$ is.)
\end{enumerate}

\noindent
The easiest to prove is $(1)$, which quickly follows from $\Ad$-invariance of $\mathcal{O}(F)$.
$\Ad$-invariance of $\mu_F$ also implies $\mu_F$ is \emph{balanced}, a notion which generalizes important features of the uniform measure.
This property of $\mu_F$, which was introduced and utilized in \cite{leake2020computability}, allows us to get bounds on the optimal solution in any direction, yielding the desired bounding box $(2)$.
See Section \ref{sec:bounding_box} for more details.

More interesting Lie theory is then needed to handle $(3)$ and $(4)$.
First, $(3)$ follows from irreducibility properties of the adjoint representation.
In particular, the convex hull of an adjoint orbit in a simple Lie algebra is either trivial or is full-dimensional.
One can then use complete reducibility for more general reductive Lie algebras, which says that the adjoint representation always decomposes into a direct sum.
These two facts in combination imply $\hull(\mathcal{O}(F))$ either spans $\mathfrak{g}_i$ or intersects $\mathfrak{g}_i$ trivially, for all $i$.
The fact that the center $\mathfrak{z}$ commutes with the adjoint action then implies the central component of $F$ remains fixed under the action.
(We were not able to find an explicit reference to the results we needed, so we prove what we need in Section \ref{sec:linear_equalities}.)

For $(4)$, we use the classical Kostant convexity theorem (Theorem \ref{thm:kostant}) which says that every (co)adjoint orbit intersects $\mathfrak{h}$ at a finite set of points, and the projection of the orbit onto $\mathfrak{h}$ is the convex hull of these (extreme) points.
Our optimization objective function then has two parts: a linear functional with respect to $A \in \mathfrak{h}$, and a $G$-invariant integral.
The fact that the adjoint orbit intersects $\mathfrak{h}$ means that we can restrict out $G$-invariant integral to taking inputs in $\mathfrak{h}$.
Optimizing a linear functional over a convex polytope is then equivalent to optimizing over the extreme points, so we can restrict the linear functional to taking inputs in $\mathfrak{h}$ as well.
We spell this out formally in Section \ref{sec:convexity}.

\begin{remark}
    We give an algorithm to solve a convex optimization problem over adjoint orbits of compact Lie groups.
    The obvious next question is: can this be extended to other classes of orbits and orbitopes?
    The most likely possibility is that of polar orbitopes, discussed in \cite{Kobert}.
    This class consists of adjoint orbits of a compact subgroup of a larger group.
    In this case, the compact group $G$ is acting on a vector space which is larger than its Lie algebra $\mathfrak{g}$.
    The corresponding orbitopes are always spectrahedra, and specific examples which fall into this class are the symmetric Schur-Horn orbitopes and the Fan orbitope (see \cite{Kobert} and \cite{sanyal2011}).
\end{remark}

\section{Proof of the Main Result} \label{sec:main_proof}

In this section, we complete the proof of Theorem \ref{thm:lie_algo}.
As stated above, the idea is to use the ellipsoid method (see Section \ref{sec:ellipsoid}) to approximate an optimal solution to the dual convex program.
To do that we need to finish the four points listed above, and the bulk of the work towards these four points is pushed to subsections.
That said, we complete the proof now, making reference to lemmas from the subsections as we need them.

Let $F := Z \oplus F_1 \oplus \cdots \oplus F_n$, and let $\pi_c: \mathfrak{g} \to \mathfrak{z}$ and $\pi_i: \mathfrak{g} \to \mathfrak{g_i}$ for all $i \in [n]$ denote the standard projections.
We want to apply Theorem \ref{thm:general_algo}, and so there are a few things to do which we listed in the previous section.

\textbf{(1) Determine a small-enough ball containing the orbit.}
Since $F$ is represented in terms of a basis orthonormal with respect to the $\Ad$-invariant inner product $\langle \cdot, \cdot \rangle$, we have that $\|F\| = \sqrt{\langle F,F \rangle} = \sqrt{\langle \Ad_g(F),\Ad_g(F) \rangle} = \|\Ad_g(F)\|$ for all $g \in G$. That is, $\mathcal{O}(F)$ is contained in a ball of radius $\|F\|$.

\textbf{(2) Determine a bounding box for the optimal solution.}
By Corollary \ref{cor:Lie_bounding_box}, the optimal solution $Y^\star$ to the dual program is such that
\[
    \|Y^\star\| \leq \frac{2d}{\eta} \log\left(\frac{8\sqrt{d}\|F\|}{\eta}\right) = \poly(d, \eta^{-1}, \log\|F\|)
\]
whenever $A$ is in the $\eta$-interior of $\mu_F$.

\textbf{(3) Determine a maximal set of affine equalities for the convex hull of the orbit.}
We appeal to Proposition \ref{prop:orbit_max_equalities} which says that
\[
    \left\{\pi_i(X) = 0 ~:~ F_i = 0\right\} \cup \left\{\pi_c(X) = Z\right\}
\]
is a maximal set of affine equalities satisfied by $\mathcal{O}(F)$.
This can be represented in matrix form as $\mathcal{L}(X) = b$, where $\mathcal{L}$ can be represented by a 0-1 matrix and the number of bits needed to represent $b$ is bounded by the number of bits needed to represent $F$.

Note that because we restrict our search space to a Cartan subalgebra in (4), we actually need to restrict the affine equalities given above to the Cartan subalgebra itself.
The simple form of the equalities given above implies this is a straightforward restriction which does not change the bit complexity.

\textbf{(4) Show that we can restrict our search space to a Cartan subalgebra.}
Note that $\mathfrak{z} \oplus \mathfrak{h}$ is a Cartan subalgebra of $\mathfrak{g}$.
This follows from the fact that $\mathfrak{z}$ is the center of $\mathfrak{g}$, and that $\mathfrak{h}_i$ is a Cartan subalgebra of $\mathfrak{g}_i$ for all $i \in [n]$.
By Proposition \ref{prop:torus}, this implies the Lie group $T \subset G$ associated to $\mathfrak{z} \oplus \mathfrak{h}$ is a maximal torus.

We first show that we can restrict our search space to $\mathfrak{z} \oplus \mathfrak{h}$ via a generalization of the Schur-Horn theorem in the unitary case.
By Theorem \ref{thm:kostant}, the adjoint orbit of any $Y \in \mathfrak{g}$ intersects $\mathfrak{z} \oplus \mathfrak{h}$.
Since
\[
    \mathcal{E}_{\mu_F}(Y) = \log \int_{\mathcal{O}(F)} e^{-\langle Y, X \rangle} d\mu_F(X)
\]
is invariant under the adjoint action of $G$ applied to $Y$, showing that $\langle Y, A \rangle$ is minimized on $\mathfrak{z} \oplus \mathfrak{h}$ implies $f_A(Y)$ is minimized on $\mathfrak{z} \oplus \mathfrak{h}$.
That $\langle Y, A \rangle$ is minimized on $\mathfrak{z} \oplus \mathfrak{h}$ then follows from Proposition \ref{prop:Kostant_optimization}, since our inner product $\langle \cdot, \cdot \rangle$ is a nondegenerate $\Ad$-invariant bilinear form on $\mathfrak{g}$.

Finally, we show that we can further restrict our search space to $\mathfrak{h}$.
To see this, note that
\[
    f_A(Y) = \log \int_{\mathcal{O}(F)} e^{\langle Y, A-X \rangle} d\mu_F(X),
\]
where $\pi_c(A-X) = 0$ since $\pi_c(X) = \pi_c(A) = Z$.
Corollary \ref{cor:center_orthogonal} then implies
\[
    \langle Y, A-X \rangle = \langle Y', A-X \rangle,
\]
where $Y' = 0 \oplus \pi_1(Y) \oplus \cdots \oplus \pi_n(Y) \in \mathfrak{h}$.
Therefore, $f_A(Y) = f_A(Y')$ for $Y' \in \mathfrak{h}$.

\begin{remark} \label{rem:center_restrict}
    Consider $V_\mathcal{L}$, the linear subspace of $\mathfrak{h}$ corresponding to the maximal set of affine equalities for $\hull(\mathcal{O}(F))$ given by
    \[
        \left\{\pi_i(X) = 0 ~:~ F_i = 0\right\} \cup \left\{\pi_c(X) = 0\right\}.
    \]
    The condition on the central component means that we only consider elements in the semisimple part of $\mathfrak{h}$.
    This shows that restricting our search space away from $\mathfrak{z}$ is actually already built in to Theorem \ref{thm:general_algo}.
\end{remark}

\subsection{Ellipsoid Framework} \label{sec:ellipsoid}

In this section we discuss the ellipsoid method, which is the underlying algorithm that the algorithm of Theorem \ref{thm:lie_algo} is based on.
We also recall the main ellipsoid method-based algorithm of \cite{leake2020computability}, which will serve as a base result for our main theorem.

First we recall the ellipsoid algorithm.
The following formulation was taken from \cite{SinghV14}, originally derived from Theorem 8.2.1 of \cite{BentalN12}.
As in Remark \ref{rem:center_restrict}, we let $V_\mathcal{L}$ denote the linear subspace of $\mathfrak{h}$ corresponding to the maximal set of affine equalities for $\hull(\mathcal{O}(F))$.

\begin{theorem}[\textbf{Ellipsoid algorithm}]\label{thm:ellipsoid}
    Given any $\beta > 0$ and $R > 0$, there is an algorithm which, given a strong first-order oracle for $f_A$, returns a $Y^\circ \in V_\mathcal{L}$ such that:
    \[
        f_A(Y^\circ) \leq \inf_{Y \in V_\mathcal{L}, \|Y\|_\infty \leq R} f_A(Y) + \beta\left(\sup_{Y \in V_\mathcal{L}, \|Y\|_\infty \leq R} f_A(Y) - \inf_{Y \in V_\mathcal{L}, \|Y\|_\infty \leq R} f_A(Y)\right).
    \]
    The number of calls to the strong first-order oracle for $f_A$ is bounded by a polynomial in $d$, $\log R$, and $\log (1/\beta)$. Here, $d$ is the dimension of the ambient Hilbert space in which $\Omega$ lies.
\end{theorem}

\noindent
We will not use the above result directly, instead using the main algorithm from \cite{leake2020computability} which is based upon the ellipsoid method.
We now state this result, which gives conditions for the existence of an algorithm for approximating the optimum to the dual objective.
This result relies on the definition of a \emph{balanced measure}, which we discuss in the next section.

\begin{theorem}[\textbf{Main algorithm from \cite{leake2020computability}}] \label{thm:general_algo}
    Let $\mu$ be a measure on a domain $\Omega \subseteq \R^d$ for which
    \begin{enumerate}
        \item $\Omega$ is contained in a ball of radius $r$,
        \item $\mu$ is balanced (see Definition \ref{def:balanced}), and
        \item $\mathcal{L}(X) = b$ is a maximal set of linearly independent equalities for $\Omega$.
    \end{enumerate}
    There exists an algorithm that, given an $A$ in the $\eta$-interior of $\mathcal{K} = \hull(\Omega)$, an $\epsilon > 0$, and a strong counting/integration oracle for the exponential integral $\mathcal{E}_\mu(Y)$, returns $Y^\circ \in V_\mathcal{L}$ such that
    \[
        f_A(Y^\circ) \leq f_A(Y^\star) + \epsilon,
    \]
    where $f_A$ is the objective function for the dual program $\dual_\mu(A)$, and $Y^\star \in V_\mathcal{L}$ is the optimum of the dual program.
    The running time of the algorithm is polynomial in $d$, $\eta^{-1}$, $\log(\epsilon^{-1})$, $\log(r)$, and the number of bits needed to represent $A$, $\mathcal{L}$, and $b$.
\end{theorem}

\subsection{Bounding Box} \label{sec:bounding_box}

To apply the ellipsoid method, we need some effective bound on the size of the optimal input for the dual program.
This was also needed in \cite{leake2020computability}, and here we are able to extend previous bounds to adjoint orbits of Lie groups.
The key notion used in \cite{leake2020computability} is that of a \emph{balanced measure}, which served as a generalization of the notion of uniform measure.
Combining arguments used in that paper with the $\Ad$-invariance properties of the measure $\mu_F$ then leads to the required bounds.

We first recall the definition of a balanced measure from \cite{leake2020computability}.

\begin{definition}[\textbf{Balanced measure}] \label{def:balanced}
    A measure $\mu$ is said to be \emph{$\delta$-balanced} if for any $X \in \Omega$ we have that at least $\exp(-\poly(\delta^{-1}, d))$ of the mass of $\mu$ is contained in the $\delta$-ball about $X$.
    A measure $\mu$ is said to be \emph{balanced} if $\mu$ is $\delta$-balanced for any $\delta > 0$.
\end{definition}

\noindent
We now prove that the invariant measure $\mu_F$ on the adjoint orbit of $F$ is balanced, which in turn implies a bounding box result for $\mu_F$.

\begin{proposition}[\textbf{Balancedness of $\Ad$-invariant measures}]
    Given a compact Lie group $G$ of dimension $d$ with corresponding $d$-dimensional Lie algebra $\mathfrak{g}$, let $\mu_F$ be the $G$-invariant probability measure on $\mathcal{O}(F)$, the adjoint orbit of $F$. Then $\mu_F$ is $\delta$-balanced with bound $f(\delta^{-1}, d) = d\log(4\sqrt{d}\|F\|\delta^{-1})$ for all $\delta > 0$.
\end{proposition}
\begin{proof}
    For all $X \in \mathcal{O}(F)$, there exists some $g \in G$ such that $X = \Ad_g(F)$.
    Therefore for all $X \in \mathcal{O}(F)$, we have
    \[
        \|X\|^2 = \langle \Ad_g(F), \Ad_g(F) \rangle = \langle F, F \rangle = \|F\|^2.
    \]
    So, $\mathcal{O}(F)$ is contained in the boundary of the ball of radius $\|F\|$ centered at 0 in $\mathfrak{g}$.
    The number of balls of size $\delta$ required to cover the Euclidean ball of radius $\|F\|$ in $\R^d$ is known to be at most $(2\|F\|\sqrt{d}/\delta)^d$.
    With this, there exists some $\delta$-ball (call in $B_\delta$) in this cover for which
    \[
        \mu_F(B_\delta) \geq (2\|F\|\sqrt{d}/\delta)^{-d}.
    \]
    Pick some $X \in \mathcal{O}(F) \cap B_\delta$, and let $B_{2\delta}(X)$ be the ball of the radius $2\delta$ which is centered at $X$. So in fact we have
    \[
        \mu_F(B_{2\delta}(X)) \geq (2\|F\|\sqrt{d}/\delta)^{-d}.
    \]
    By $G$-invariance of $\mu_F$, this bound applies to any $X \in \mathcal{O}(F)$.
    That is, $\mu_F$ is $(2\delta)$-balanced with bound
    \[
        f((2\delta)^{-1}, d) = d \cdot \log(2\sqrt{d}\|F\|\delta^{-1})
    \]
    for all $\delta > 0$. The result follows.
\end{proof}

\noindent
We next recall the general bounding box result from \cite{leake2020computability} and use it to prove a bounding box for $\mu_F$.
Again, recall that we let $V_\mathcal{L}$ denote the linear subspace corresponding to the maximal set of affine equalities for $\hull(\mathcal{O}(F))$.

\begin{theorem}[\textbf{Bounding box from \cite{leake2020computability}}]
    Suppose $\mu$ is $\frac{\eta}{2}$-balanced with bound $f$. If $A$ is in the $\eta$-interior of $\mu$ and $Y^\star \in V_\mathcal{L}$ is the optimal solution to the corresponding dual program, then $\|Y^\star\| \leq 2\eta^{-1} \cdot f(2\eta^{-1}, d) = \poly(\eta^{-1}, d)$.
\end{theorem}

\begin{corollary}[\textbf{Bounding box for adjoint orbits}] \label{cor:Lie_bounding_box}
    Fix a compact Lie group $G$ of dimension $d$ with corresponding $d$-dimensional Lie algebra $\mathfrak{g}$, and fix $A,F \in \mathfrak{g}$ such that $A$ is in the $\eta$-interior of $\mu_F$. If $Y^\star \in V_\mathcal{L}$ is the optimal solution to the dual program $\dual_F(A)$, then $\|Y^\star\| \leq \frac{2d}{\eta} \log\left(\frac{8\sqrt{d}\|F\|}{\eta}\right)$.
\end{corollary}

\subsection{Affine Equalities} \label{sec:linear_equalities}

In this section, we discuss $\ad$-invariant subspaces of $\mathfrak{g}$ and their connections to affine equalities for adjoint orbits.
A subspace $V \subseteq \mathfrak{g}$ is said to be $\ad$-invariant if $\ad_X(Y) \in V$ for all $X \in \mathfrak{g}$ and $Y \in V$.
The results stated here are mainly things that are well-known to people who work within Lie theory, and more or less follow from reducibility properties of the adjoint representation.
We could not find explicit references to these results, and so we have proven them here.

The main idea behind the first result is that the $\Ad$ and $\ad$ actions cut out the same subspaces.
Using this, we can then apply affine equalities for the orbit of the action of $\ad$ to the adjoint orbit.
This is helpful because the action of $\ad$ is easier to handle in this context.

\begin{lemma} \label{lem:Ad_vs_ad}
    Let $G$ be a compact connected Lie group with corresponding Lie algebra $\mathfrak{g}$.
    Fix $F \in \mathfrak{g}$, and define $V := \spn(\mathcal{O}(F))$ where $\mathcal{O}(F)$ is the adjoint orbit of $F$.
    Then $V$ is the minimal $\ad$-invariant subspace of $\mathfrak{g}$ containing $F$.
\end{lemma}
\begin{proof}
    Let $W$ be the minimal $\ad$-invariant subspace of $\mathfrak{g}$ containing $F$. We first show that $W \subseteq V$. For any $X \in \mathfrak{g}$ and $g \in G$, Corollary \ref{cor:exp_deriv} implies
    \[
        \ad_X(\Ad_g(F)) = \left.\frac{d}{dt} \Ad_{\exp(tX)}(\Ad_g(F))\right|_{t=0} = \left.\frac{d}{dt} \Ad_{\exp(tX)g}(F)\right|_{t=0} \in V,
    \]
    where the fact that this is always in $V$ follows from the fact that $\Ad_{\exp(tX)g}(F) \in V$ for all $t$. In general for any $Y \in V$, we have $Y = \sum_i c_i \Ad_{g_i}(F)$, which implies
    \[
        \ad_X(Y) = \sum_i c_i \ad_X(\Ad_{g_i}(F)) \in V.
    \]
    That is, $\ad_X(Y) \in V$ for all $X \in \mathfrak{g}$ and $Y \in V$, and so $W \subseteq V$.

    We now show that $V \subseteq W$. Surjectivity of $\exp$ (Proposition \ref{prop:Lie_exp}) implies that for any $g \in G$ there is an $X \in \mathfrak{g}$ such that $g = \exp(X)$. From this we have
    \[
        \Ad_g(F) = \Ad_{\exp(X)}(F) = \exp(\ad_X)(F) = \sum_{n=0}^\infty \frac{\ad_X^n(F)}{n!} \in W,
    \]
    where the fact this is always in $W$ follows from the fact that $W$ is an $\ad$-invariant subspace of $\mathfrak{g}$ containing $F$. Since $V$ is the span of all such elements $\Ad_g(F)$, we further have that $V \subseteq W$.
\end{proof}

\noindent
Recall that a simple Lie algebra $\mathfrak{g}$ is one for which there are only two $\ad$-invariant subspaces, $V = 0$ and $V = \mathfrak{g}$.
This can be generalized to reductive Lie algebras in the following result.

\begin{proposition}
    Let $\mathfrak{g}$ be a real reductive Lie algebra with decomposition into center and simple components given by $\mathfrak{g} = \mathfrak{z} \oplus \mathfrak{g}_1 \oplus \cdots \oplus \mathfrak{g}_n$. Every $\ad$-invariant subspace $V \subseteq \mathfrak{g}$ is of the form $V = W \oplus V_1 \oplus \cdots \oplus V_n$, where $W$ is any subspace of $\mathfrak{z}$ and $V_i = \{0\}$ or $V_i = \mathfrak{g}_i$ for all $i \in [n]$.
\end{proposition}
\begin{proof}
    Let $V$ be $\ad$-invariant, and let $F = Z \oplus F_1 \oplus \cdots \oplus F_n$ be an element of $V$ for which $S := \{i \in [n] ~:~ F_i \neq 0\}$ is the unique maximum with respect to the inclusion order over all such $F$. (Note that this maximum is unique because the base field is infinite.) Now fix any $i \in S$, and let $\iota_i: \mathfrak{g}_i \to \mathfrak{g}$ denote the natural inclusion. Since $\mathfrak{g}_i$ is simple and $F_i \neq 0$, we can pick $X_i \in \mathfrak{g}_i$ such that $\ad_{X_i}(F_i) \neq 0$ and $0 \neq \ad_{\iota_i(X_i)}(F) = \iota_i(\ad_{X_i}(F_i)) \in V$. By simplicity of $\mathfrak{g}_i$, this implies $\iota_i(\mathfrak{g}_i) \subseteq V$. The result then follows from the maximality of $S$.
\end{proof}

\noindent
We now apply this to spans of adjoint orbits.

\begin{corollary} \label{cor:ad_subspace}
    Let $G$ be a compact real Lie group with reductive Lie algebra $\mathfrak{g} = \mathfrak{z} \oplus \mathfrak{g}_1 \oplus \cdots \oplus \mathfrak{g}_n$, and fix any $F = Z \oplus F_1 \oplus \cdots \oplus F_n \in \mathfrak{g}$. Then the minimal $\ad$-invariant subspace of $\mathfrak{g}$ containing $F$ is $\spn\{Z\} \oplus V_1 \oplus \cdots \oplus V_n$ where for all $i \in [n]$, $V_i = \mathfrak{g_i}$ if $F_i \neq 0$ and $V_i = \{0\}$ otherwise.
\end{corollary}
\begin{proof}
    Since $\ad_X(Z) = 0$ for every $X \in \mathfrak{z}$ ($\mathfrak{z}$ is Abelian), the result follows from the previous proposition.
\end{proof}

\noindent
We next use these facts regarding $\ad$-invariant subspaces to determine a maximal set of affine equalities for $\hull(\mathcal{O}(F))$.
First, we need a lemma.

\begin{lemma} \label{lem:center_action}
    Let $G$ be a compact connected real Lie group with reductive Lie algebra $\mathfrak{g} = \mathfrak{z} \oplus \mathfrak{g}_0$ where $\mathfrak{z}$ is the center of $\mathfrak{g}$. For any $F = Z \oplus X \in \mathfrak{z} \oplus \mathfrak{g}_0$ and any $g \in G$, we have that $\Ad_g(F) = Z \oplus X'$ for some $X' \in \mathfrak{g}_0$.
\end{lemma}
\begin{proof}
    Let $\pi_c: \mathfrak{g} \to \mathfrak{z}$ be the canonical projection map. Since $G$ is compact connected, for all $g \in G$ there is an $X \in \mathfrak{g}$ such that $g = \exp(X)$ (by Proposition \ref{prop:Lie_exp}). This implies
    \[
        \Ad_g(F) = \exp(\ad_X)(F) = \sum_{n=0}^\infty \frac{\ad_X^n(F)}{n!} = F + \sum_{n=0}^\infty \frac{\ad_X(\ad_X^n(F))}{(n+1)!}.
    \]
    Since $\pi_c(\ad_X(Y)) = 0$ for all $Y \in \mathfrak{g}$, this further implies $\pi_c(\Ad_g(F)) = \pi_c(F) = Z$ for all $g \in G$.
\end{proof}

\begin{proposition}[\textbf{Maximal set of affine equalities}] \label{prop:orbit_max_equalities}
    Let $G$ be a compact connected real Lie group with reductive Lie algebra $\mathfrak{g} = \mathfrak{z} \oplus \mathfrak{g}_1 \oplus \cdots \oplus \mathfrak{g}_n$, and fix any $F = Z \oplus F_1 \oplus \cdots \oplus F_n \in \mathfrak{g}$. If $\pi_c,\pi_1,\ldots,\pi_n$ are the natural projections of $\mathfrak{g}$ onto $\mathfrak{z}, \mathfrak{g}_1, \ldots, \mathfrak{g}_n$ respectively, then
    \[
        \left\{\pi_i(X) = 0 ~:~ F_i = 0\right\} \cup \left\{\pi_c(X) = Z\right\}
    \]
    is a maximal set of affine equalities satisfied by $\mathcal{O}(F)$.
\end{proposition}
\begin{proof}
    Combining Lemma \ref{lem:Ad_vs_ad} with Corollary \ref{cor:ad_subspace} implies $\spn(\mathcal{O}(F)) = \spn\{Z\} \oplus V_1 \oplus \cdots \oplus V_n$ where for all $i \in [n]$, $V_i = \mathfrak{g}_i$ if $F_i \neq 0$ and $V_i = \{0\}$ otherwise.
    Therefore the affine span of $\mathcal{O}(F)$ is of dimension at least $\dim(\spn\{Z\} \oplus V_1 \oplus \cdots \oplus V_n) - 1$.
    
    First suppose $Z \neq 0$. By Lemma \ref{lem:center_action}, we then have that $\pi_c(\Ad_g(F)) = \pi_c(F) = Z$ for all $g \in G$. Therefore the affine span of $\mathcal{O}(F)$ satisfies the set of affine equalities specified above, and this set is maximal by the dimension property of $\mathcal{O}(F)$ mentioned above.

    Now suppose $Z = 0$. So as to get a contradiction, suppose that there is an additional affine equality satisfied by $\mathcal{O}(F)$, given by $B(\Ad_g F, Y) = c$ where $B$ is the Killing form, $Y$ is a non-central element of $\mathfrak{g}$, and $c \in \R$. By $\Ad$-invariance of $B$, we further have that
    \[
        B(\Ad_g F, \Ad_h Y) = c
    \]
    for all $g,h \in G$. Since $Y \not\in \mathfrak{z}$, there is some $h \in G$ such that $\Ad_h Y \neq Y$. This implies
    \[
        B(\Ad_g F, Y - \Ad_h Y) = 0
    \]
    for all $g \in G$. That is, $\mathcal{O}(F)$ is contained in a subspace of $\mathfrak{g}$, a contradiction. Therefore, the above set of affine equalities is maximal.
\end{proof}

\noindent
The last results of this section then show how the maximal set of affine equalities implies orthogonality of the center of $\mathfrak{g}$ with respect to any $\Ad$-invariant bilinear form.
Combining this with Proposition \ref{prop:Killing_orthogonal} then characterizes all possible $\Ad$-invariant inner products which extend the negative of the Killing form.

\begin{corollary}[\textbf{Center orthogonality}] \label{cor:center_orthogonal}
    Let $B$ be an $\Ad$-invariant symmetric bilinear form on a real reductive Lie algebra $\mathfrak{g} = \mathfrak{z} \oplus \mathfrak{g}_0$, where $\mathfrak{z}$ is the center of $\mathfrak{g}$ and $\mathfrak{g}_0$ is the semisimple part of $\mathfrak{g}$. Then for all $Z \in \mathfrak{z} \oplus \{0\}$ and $X \in \{0\} \oplus \mathfrak{g}_0$ we have $B(Z,X) = 0$.
\end{corollary}
\begin{proof}
    Let $\mathfrak{g} = \mathfrak{z} \oplus \mathfrak{g}_1 \oplus \cdots \oplus \mathfrak{g}_n$ be the decomposition of $\mathfrak{g}$ into its center and simple components, and let $\pi_c,\pi_1,\ldots,\pi_n$ be the projections onto each component. By Proposition \ref{prop:orbit_max_equalities}, we have that
    \[
        \{\pi_i(Y) = 0 ~:~ \pi_i(X) = 0\} \cup \{\pi_c(Y) = 0\}
    \]
    is a maximal set of affine equalities satisfied by $\mathcal{O}(X)$. Further by Lemma \ref{lem:center_action}, for all $g \in G$ we have
    \[
        B(Z, X) = B(\Ad_g Z, \Ad_g X) = B(Z, \Ad_g X),
    \]
    which implies $B(Z, Y) = B(Z, X)$ for all $Y \in \mathcal{O}(X)$. This is an affine equality satisfied by $\mathcal{O}(X)$, and so it must be implied by the those listed above. So in fact for any $Y$ which lies in the subspace of $\mathfrak{g}$ cut out by the above set of affine equalities, we have that $B(Z, Y) = B(Z, X)$. That is, the linear functional $B(Z, \cdot)$ is constant on a vector subspace and therefore must be 0 on that subspace. Therefore $B(Z, X) = 0$.
\end{proof}

\begin{corollary}[\textbf{Killing form inner product}] \label{cor:extend_Killing}
    Let $\langle \cdot, \cdot \rangle$ be an $\Ad$-invariant inner product on a real reductive Lie algebra $\mathfrak{g} = \mathfrak{z} \oplus \mathfrak{g}_0$ such that for all $X,Y \in \mathfrak{g}_0$ we have
    \[
        \langle X, Y \rangle = -B(X, Y),
    \]
    where $\mathfrak{z}$ is the center of $\mathfrak{g}$, $\mathfrak{g}_0$ is the semisimple part of $\mathfrak{g}$, and $B$ is the Killing form of $\mathfrak{g}$.
    Then $\langle \cdot, \cdot \rangle$ is $\Ad$-invariant if and only if $\mathfrak{z}$ and $\mathfrak{g}_0$ are orthogonal subspaces.
\end{corollary}
\begin{proof}
    $(\implies)$. Corollary \ref{cor:center_orthogonal}.

    $(\impliedby)$. Since the Killing form $B$ is a direct sum of the Killing form on $\mathfrak{g}_0$ and the zero form on $\mathfrak{z}$, for $X_1, X_2 \in \mathfrak{g}$ we can write
    \[
        \langle X_1, X_2 \rangle = \langle \pi_c(X_1), \pi_c(X_2) \rangle - B(X_1, X_2)
    \]
    where $\pi_c$ is the projection map $\mathfrak{g} \to \mathfrak{z}$.
    Applying the $\Ad$ action for $g \in G$ and using Lemma \ref{lem:center_action} then implies
    \[
    \begin{split}
        \langle \Ad_g X_1, \Ad_g X_2 \rangle &= \langle \pi_c(\Ad_g(X_1)), \pi_c(\Ad_g(X_2)) \rangle - B(\Ad_g(X_1), \Ad_g(X_2)) \\
            &= \langle \pi_c(X_1), \pi_c(X_2) \rangle - B(X_1, X_2) \\
            &= \langle X_1, X_2 \rangle.
    \end{split}
    \]
    That is, $\langle \cdot, \cdot \rangle$ is $\Ad$-invariant.
\end{proof}

\subsection{Restrict to Cartan subalgebra} \label{sec:convexity}

In this section, we prove the main result needed to show that we can restrict to a Cartan subalgebra of $\mathfrak{g}$.
This is essentially a corollary of the Kostant convexity theorem, which is itself a generalization of the Schur-Horn theorem describing the diagonals of matrices with prescribed eigenvalues.
We first discuss Cartan subalgebras, which generalize the notion of diagonal matrices.

Let $G$ be a compact group, and let $\mathfrak{g}$ be its (reductive) Lie algebra.
A \emph{torus} is a compact connected Abelian Lie subgroup of $T \subset G$, and a \emph{maximal torus} is a torus which is not contained in any other torus.
For example, the set of all diagonal matrices in $GL_n(\C)$ or in $\mathrm{U}(n)$ is a maximal torus.
Similarly, a \emph{maximal Abelian subalgebra} is an Abelian subalgebra $\mathfrak{t} \subset \mathfrak{g}$ which is not contained in any other Abelian subalgebra.
As discussed in the preliminaries, we also refer to such subalgebras as \emph{Cartan subalgebras}.
The connection between tori and Abelian subalgebras is given in the following.

\begin{proposition}[\cite{knapp2013}, Section 1.10 and Proposition 4.30] \label{prop:torus}
    Let $G$ be a compact connected Lie group, and let $\mathfrak{g}$ be its Lie algebra.
    If $\mathfrak{t} \subset \mathfrak{g}$ is an Abelian subalgebra, then $T = \exp(\mathfrak{t}) \subset G$ is a torus.
    Further, if $\mathfrak{t}$ is maximal, then so is $T$.
\end{proposition}

We now state a classical result of Kostant regarding coadjoint orbits, which is the dual notion to adjoint orbits.
We do not further discuss this notion here, as we will only need it for this section.
Further, an $\Ad$-invariant inner product on $\mathfrak{g}$ (which we always have in the compact case) allows us to transfer any results on coadjoint orbits to analogous results on adjoint orbits.

\begin{theorem}[\textbf{Kostant convexity theorem \cite{Kostant1973}}; see also \cite{ziegler1992kostant}] \label{thm:kostant}
    Let $G$ be a compact connected Lie group, $T$ a maximal torus, $\mathfrak{g}$ and $\mathfrak{t}$ the associated Lie algebras, and $\pi: \mathfrak{g}^* \to \mathfrak{t}^*$ the natural projection.
    Then every coadjoint orbit $\mathcal{O}(X^*)$ of $G$ intersects $\mathfrak{t}^*$ in a Weyl group orbit $\Omega(X^*)$, and $\pi(\mathcal{O}(X^*)) = \hull(\Omega(X^*))$.
    Given a nondegenerate $\Ad$-invariant bilinear form on $\mathfrak{g}$, this result can be transferred to adjoint orbits in $\mathfrak{g}$.
\end{theorem}

\noindent
The Kostant convexity theorem then implies that we can restrict to a maximal Abelian (Cartan) subalgebra, as follows.

\begin{proposition}[\textbf{Restrict to Cartan subalgebra}] \label{prop:Kostant_optimization}
    Let $G$, $T$, $\mathfrak{g}$, $\mathfrak{t}$ be as in the Kostant convexity theorem with $\mathfrak{t} \subset \mathfrak{g}$, and suppose $\mathfrak{g}$ has an $\Ad$-invariant inner product $\langle \cdot, \cdot \rangle$.
    Then for any $X,Z \in \mathfrak{t}$, we have
    \[
        \min_{Y \in \mathcal{O}(X)} \langle Y, Z \rangle = \min_{Y \in \Omega(X)} \langle Y, Z \rangle,
    \]
    where $\mathcal{O}(X)$ is the adjoint orbit of $X$ and $\Omega(X)$ is the Weyl group orbit of $X$.
    That is, optimization of a linear functional can be restricted to $\Omega(X) = \mathcal{O}(X) \cap \mathfrak{t}$ in this case.
\end{proposition}
\begin{proof}
    Let $\pi^*: \mathfrak{g}^* \to \mathfrak{t}^*$ be the natural projection, which is defined to be dual to the inclusion $\iota: \mathfrak{t} \to \mathfrak{g}$ via
    \[
        [\pi^*(X^*)](Y) := X^*(\iota(Y)) = X^*(Y)
    \]
    for $X^* \in \mathfrak{g}^*$ and $Y \in \mathfrak{t}$.
    We now use the invariant inner product to get an isomorphism between $\mathfrak{g}$ and $\mathfrak{g}^*$, which we denote by $X \mapsto X^*$ and define
    \[
        X^*(Y) := \langle X, Y \rangle
    \]
    for all $Y \in \mathfrak{g}$.
    Let $\pi: \mathfrak{g} \to \mathfrak{t}$ now denote the orthogonal projection, which corresponds to $\pi^*$ via this isomorphism.
    This implies
    \[
        \langle \pi(X), Y \rangle = \langle X, \iota(Y) \rangle = \langle X, Y \rangle
    \]
    for all $X \in \mathfrak{g}$ and $Y \in \mathfrak{t}$.
    
    Since $\langle \cdot, \cdot \rangle$ is $\Ad$-invariant, we can apply the Kostant convexity theorem to the adjoint orbit $\mathcal{O}(X)$ to get $\pi(\mathcal{O}(X)) = \hull(\Omega(X))$.
    And since we are optimizing a linear functional, the maximum value occurs at an extreme point.
    Therefore,
    \[
        \min_{Y \in \Omega(X)} \langle Y, Z \rangle = \min_{Y \in \hull(\Omega(X))} \langle Y, Z \rangle = \min_{Y \in \pi(\mathcal{O}(X))} \langle Y, Z \rangle = \min_{Y \in \mathcal{O}(X)} \langle \pi(Y), Z \rangle = \min_{Y \in \mathcal{O}(X)} \langle Y, Z \rangle.
    \]
\end{proof}

\section{Strong Counting Oracle via the Harish-Chandra Formula} \label{sec:HC}

Here we discuss some general principles of the formulas which give rise to strong counting/integration oracles for various families of groups.
We leave discussion of explicit examples for specific groups to Section \ref{sec:applications}.
In general, all formulas we achieve come by way of the Harish-Chandra integral formula, which we state now.
(Note that the following result contains some notation that we have not defined.
We will define and discuss what is important to our applications, and we refer the reader to \cite{mcswiggen2018harish} for any undefined notation.)

\begin{theorem}[\textbf{Harish-Chandra integral formula \cite{HarishChandra1957}}] \label{thm:HC}
    Let $G$ be a compact connected semisimple Lie group, $\mathfrak{g}$ the associated Lie algebra, and $\mathfrak{h} \subset \mathfrak{g}$ a Cartan subalgebra.
    For any $h_1,h_2 \in \mathfrak{h}$ we have
    \[
        \int_G e^{B(h_1, \Ad_g h_2)} dg = C_G \cdot \frac{\sum_{w \in W} \epsilon(w) e^{B(h_1, w(h_2))}}{\Pi(h_1) \Pi(h_2)},
    \]
    where $B$ is a Killing form of $\mathfrak{g}$, $dg$ is the Haar measure on $G$, $W$ is the Weyl group of $G$, $\Pi$ is an efficiently computable function on $\mathfrak{h}$, and $C_G$ is an efficiently computable constant.
\end{theorem}

\noindent
Note that this extends upon the well-known HCIZ integral formula (see Theorem \ref{thm:HCIZ} below) because it allows for any compact connected semisimple Lie group, and not just the unitary group.
On the other hand it is not immediately clear how useful this formula is to us, as the above sum is over the (at least) exponentially-sized Weyl group $W$.
In Section \ref{sec:applications}, we discuss how to convert the large sum into small determinants for specific groups.
Before doing this, we use the rest of this section to discuss various issues related to turning the above Harish-Chandra integral formula into a strong counting oracle for adjoint orbits.

\subsection{The Weyl group}

Associated to every Lie group $G$ is a particular finite group $W$ called its \emph{Weyl group}.
There are many equivalent ways to define the Weyl group of a given Lie group, but we describe none of them here.
(The interested reader can consult any standard Lie theory reference.)
In this paper, the most important property of the Weyl group is its appearance in the Harish-Chandra integral formula above.

Conceptually, $W$ is a ``quadrature group'' for $G$, which is one way to interpret the Harish-Chandra formula.
That is, integrals of certain exponential functions over the group $G$ (a manifold) can be computed by instead summing those functions over the Weyl group $W$ (a discrete set).
Unfortunately the Weyl group $W$ is often related to the symmetric group, and so it is a priori too large for the Harish-Chandra formula to be efficiently computable.
For many matrix groups however, the formula can be rewritten in terms of a small sum of determinants, leading to something which can be effectively computed.
We discuss this further for specific examples in Section \ref{sec:applications}.

\subsection{Integrals on orbits}

A priori, the above integral is on the group $G$ and not on the orbit itself.
We now show how the Harish-Chandra formula actually gives a formula for computing the exponential integral we care about, given by
\[
    \mathcal{E}_F(Y) = \log \int_{\mathcal{O}(F)} e^{-\langle Y, X \rangle} d\mu_F(X)
\]
for $F \in \mathfrak{z} \oplus \mathfrak{h}$, where $\langle \cdot, \cdot \rangle$ is $\Ad$-invariant.
First, we can restrict our search space to $Y \in \mathfrak{h}$ by Theorem \ref{thm:lie_algo}.
By Corollary \ref{cor:center_orthogonal}, we may then write
\[
    \mathcal{E}_F(Y) = \mathcal{E}_{F'}(Y) = \log \int_{\mathcal{O}(F')} e^{-\langle Y, X \rangle} d\mu_{F'}(X),
\]
where $F' \in \mathfrak{h}$ is the projection $\mathfrak{z} \oplus \mathfrak{h} \to \mathfrak{h}$ applied to $F$.
We now use a standard fact about integrals on groups to transfer the integral of $\mathcal{E}_F(Y)$ to the group $G$.

\begin{theorem}[\cite{helgason1984}, Theorem 1.9] \label{thm:helgason}
    Let $H$ be a closed subgroup of a compact Lie group $G$. Then there exists a unique $G$-invariant measure on $G/H$ such that for any continuous $f:G \to \R$ we have:
    \[
        \int_G f(g) dg = \int_{G/H} \left(\int_H f(gh) dh\right) dg_H.
    \]
    Here $dg$ and $dh$ are the respective Haar probability measures, and $dg_H$ is the $G$-invariant measure on $G/H$.
\end{theorem}

\noindent
Letting $\Phi: G / \Stab(F') \to \mathcal{O}(F')$ be the standard isomorphism given by $\Phi(g \cdot \Stab(F')) = \Ad_{g} F'$ for all $g \in G$, we can consider $\mu_{F'}$ to be supported on $G / \Stab(F')$.
By uniqueness, $\mu_{F'}$ is the measure on $G / \Stab(F')$ claimed by the theorem (up to scalar).
With this we compute
\[
\begin{split}
    \mathcal{E}_{F'}(Y) &= \log \int_{G / \Stab(F')} e^{-\langle Y, \Ad_g F' \rangle} d\mu_{F'}(g \cdot \Stab(F')) \\ 
        &= \log \int_{G / \Stab(F')} \left(\int_{\Stab(F')} e^{-\langle Y, \Ad_{gh} F' \rangle} dh\right) d\mu_{F'}(g \cdot \Stab(F')) \\
        &= \log \int_G e^{-\langle Y, \Ad_g F' \rangle} dg.
\end{split}
\]
That is, we have
\[
    \mathcal{E}_F(Y) = \log \int_{\mathcal{O}(F)} e^{-\langle Y, X \rangle} d\mu_F(X) = \log \int_G e^{-\langle Y, \Ad_g F' \rangle} dg.
\]
With this, we can apply the Harish-Chandra formula to compute $\mathcal{E}_F(Y)$.

\subsection{Other computability issues}

Now, the Harish-Chandra formula gives \emph{some} algorithm for computing $\mathcal{E}_F(Y)$, but the problem that still remains is that the size of the Weyl group $W$ is usually larger than exponential in the dimension of $\mathfrak{g}$.
Fortunately in most of the interesting cases, the exponentially sized sum in the formula becomes a small sum of ratios of determinants (see Section \ref{sec:applications} for specific examples).
An example to keep in mind is the HCIZ integral formula (Theorem \ref{thm:HCIZ}):
\[
    \int_{\mathrm{U}(n)} e^{\langle A, UBU^* \rangle} dU = \left(\prod_{p=1}^{n-1} p!\right) \frac{\det(e^{a_ib_j})_{i,j=1}^n}{\prod_{i<j} (a_j-a_i)(b_j-b_i)}.
\]
Multiplicity issues (i.e., when two values of $a_i$ or $b_j$ are equal) are then the last things to deal with to obtain an efficiently computable formula for such integrals.
This is handled in a standard way, as in \cite{leake2020computability}.
First, in all the cases we consider, the $\Pi(h_i)$ factors in the Harish-Chandra formula can be written as certain Vandermonde-like determinants.
Then, since we know that the integral itself is a smooth function of its inputs, we can apply L'Hoptial's rule to the ratio of determinants.
In all the cases we consider, the matrices involved are such that only one row or column depends on each variable involved.
Multilinearity of the determinant then means we can pass these derivatives directly to appropriate row or column, and so the resulting expression is still a sum of a small number of ratios of determinants.

To obtain the strong counting oracle required by \ref{thm:lie_algo}, we also need to be able to compute the gradient of $\mathcal{E}_F(Y)$.
Generally speaking, the same observation utilized to handle multiplicity can be used to compute the partial derivatives.
That is, multilinearity of the determinant means we can pass the partial derivative directly onto the appropriate row or column of the matrices involved.
What remains is to handle multiplicity in the expressions for the gradient.
What is different in this case is that these expressions may now contain ratios of \emph{products} of determinants:
\[
    \partial_x \frac{\det(A)}{\det(B)} = \frac{\det(\partial_x A)\det(B) - \det(A) \cdot \det(\partial_x B)}{\det(B)^2}.
\]
When applying L'Hopital's rule to this expression, it is a priori possible that the number of terms in the numerator will blow up exponentially as derivatives are applied.
We now argue why this can't actually happen, by considering a toy example where we are limiting $x \to y$.
Let us first write
\[
    \det(A) = O((x-y)^j) \qquad \text{and} \qquad \det(B) = O((x-y)^k),
\]
where we know by continuity of $\frac{\det(A)}{\det(B)}$ that $j \geq k$. We then have
\[
\begin{split}
    \lim_{x \to y} \partial_x \frac{\det(A)}{\det(B)} &= \lim_{x \to y} \frac{\det(\partial_x A)\det(B) - \det(A) \cdot \det(\partial_x B)}{\det(B)^2} \\
        &= \lim_{x \to y} \frac{\partial_x^{2k}[\det(\partial_x A)\det(B) - \det(A) \cdot \det(\partial_x B)]}{\partial_x^{2k}[\det(B)^2]} \\
        &= \lim_{x \to y} \frac{\sum_{j=0}^{2k} \binom{2k}{j} \det(\partial_x^{j+1} A) \det(\partial_x^{2k-j} B) - \sum_{j=0}^{2k} \binom{2k}{j} \det(\partial_x^j A) \det(\partial_x^{2k-j+1} B)}{\det(\partial_x^k B)^2} \\
        &= \lim_{x \to y} \frac{[\binom{2k}{k} - \binom{2k}{k-1}] \cdot [\det(\partial_x^{k+1} A) \det(\partial_x^k B) - \det(\partial_x^k A) \det(\partial_x^{k+1} B)]}{\det(\partial_x^k B)^2} \\
        &= \left.\binom{2k}{k} \cdot \frac{\det(\partial_x^{k+1} A) \det(\partial_x^k B) - \det(\partial_x^k A) \det(\partial_x^{k+1} B)}{(k+1) \cdot \det(\partial_x^k B)^2}\right|_{x=y}.
\end{split}
\]
The reason the above limit expression is correct is because we must apply at least $k$ derivatives to each of $\det(A)$ and $\det(B)$ in order to get a non-zero term in the numerator.
Specifically, this shows that we still have a sum of a small number of ratios of determinants.
Applying this type of argument to handle more general multiplicity then gives an efficient formula for computing the gradient of $\mathcal{E}_F(Y)$.

\section{Applications} \label{sec:applications}

In this section, we discuss examples of compact connected Lie groups to which we can apply our framework for computing maximum entropy distributions on orbits.
The particular groups we discuss are the main classical building blocks for the theory of compact Lie groups, and so they will cover a large portion of all possibilities.
For each such group, we will discuss a few important features of the Lie group $G$ and its corresponding Lie algebra $\mathfrak{g}$ which allow us to give efficient algorithms for maximum entropy optimization.
In particular, we will discuss
\begin{enumerate}
    \item An $\Ad$-invariant inner product $\langle \cdot, \cdot \rangle$ on $\mathfrak{g}$,
    \item A maximal torus $T \subset G$ and the corresponding Cartan subalgebra $\mathfrak{h} \subset \mathfrak{g}$,
    \item A basis of $\mathfrak{g}$ which is orthonormal with respect to $\langle \cdot, \cdot \rangle$ and restricts to an orthonormal basis on $\mathfrak{h}$,
    \item A counting oracle for the exponential integral $\mathcal{E}_F(Y)$ on an adjoint orbit $\mathcal{O}(F) \subset \mathfrak{g}$.
\end{enumerate}

\noindent
The bulk of the work of this section is towards demonstrating that the Harish-Chanrda formula (Theorem \ref{thm:HC}) gives a counting oracle for the exponential integral.
This relies heavily on the recent work of \cite{mcswiggen2018harish} which gives explicit determinantal formulas for a large class of compact Lie groups.

\subsection{Compact matrix groups in general}

All groups that we will consider are compact matrix groups.
Because of this, the $\Ad$-invariant inner products we consider will always be restrictions of $\langle X, Y \rangle = -\Tr(XY)$ to subalgebras of $\mathfrak{u}(n)$.
Note that the negative sign is needed here because we the matrices in $\mathfrak{u}(n)$ are skew-Hermitian.
To see further that this is $\Ad$-invariant, fix any $g \in \mathrm{U}(n)$ and $X,Y \in \mathfrak{u}(n)$ and write
\[
    \langle \Ad_g X, \Ad_g Y \rangle = -\Tr(gXg^{-1} gYg^{-1}) = -\Tr(XY) = \langle X, Y \rangle.
\]
When restricted to the semisimple part of $\mathfrak{u}(n)$, this inner product is also equal to the negative of the Killing form (up to scalar).
This fact will enable us to apply the Harish-Chandra formula (Theorem \ref{thm:HC}) for the specific groups we consider here.

\subsection{Unitary groups}

The unitary groups $\mathrm{U}(n)$ and $\mathrm{SU}(n)$ are defined as
\[
    \mathrm{U}(n) := \{U \in GL_n(\C) ~:~ U^*U = I\} \quad \text{and} \quad \mathrm{SU}(n) := \{U \in SL_n(\C) ~:~ U^*U = I\}.
\]
These are real compact connected Lie groups of dimension $n^2$ and $n^2-1$, respectively.
The associated Lie algebras are given as
\[
    \mathfrak{u}(n) := \{X \in \mathfrak{gl}_n(\C) ~:~ X \text{ is skew-Hermitian}\} \quad \text{and} \quad \mathfrak{su}(n) := \{X \in \mathfrak{sl}_n(\C) ~:~ X \text{ is skew-Hermitian}\}.
\]
The Lie algebra $\mathfrak{su}(n)$ has trivial center, while the Lie algebra $\mathfrak{u}(n)$ has center consisting of all valid multiples of the identity matrix.
Further, the Lie algebra of $\mathfrak{su}(n)$ is actually simple, which means that the Killing form on $\mathfrak{su}(n)$ is the unique invariant symmetric bilinear form up to scalar.
So it must be a multiple of the Frobenius inner product.
The fact that $\mathrm{SU}(n)$ is compact implies $B$ is negative semidefinite, but this can also be seen from the fact that $\mathfrak{su}(n)$ consists solely of skew-Hermitian matrices.

We choose the standard maximal torus $T$ given by the diagonal matrices of $\mathrm{U}(n)$ or $\mathrm{SU}(n)$, which corresponds to the Cartan subalgebra $\mathfrak{h}$ consisting of diagonal matrices of $\mathfrak{u}(n)$ or $\mathfrak{su}(n)$.
An orthonormal basis for $\mathfrak{h}$ and $\mathfrak{g}$ is then given by
\[
    H_j := i E_{j,j} \quad \text{for } j \in [n], \qquad X_{j,k} := \frac{i (E_{j,k} + E_{k,j})}{\sqrt{2}} \quad \text{for } j < k, \qquad Y_{j,k} := \frac{E_{j,k} - E_{k,j}}{\sqrt{2}} \quad \text{for } j < k.
\]
The following well-known HCIZ formula, which is a specific case of the Harish-Chandra formula (Theorem \ref{thm:HC}), then gives rise to a counting oracle for orbits of these groups.

\begin{theorem}[\textbf{Harish-Chandra-Itzykson-Zuber integral formula \cite{HarishChandra1957,IZ1980}}; see also \cite{mcswiggen2018harish}] \label{thm:HCIZ}
    Given diagonal matrices $A = \sum_j a_j H_j$ and $B = \sum_j b_j H_j$ in $\mathfrak{h}$, we have
    \[
        \int_{\mathrm{U}(n)} e^{\langle A, UBU^* \rangle} dU = \int_{\mathrm{SU}(n)} e^{\langle A, UBU^* \rangle} dU = \left(\prod_{p=1}^{n-1} p!\right) \frac{\det(e^{a_ib_j})_{i,j=1}^n}{\prod_{i<j} (a_j-a_i)(b_j-b_i)}.
    \]
\end{theorem}

\subsection{Orthogonal groups}

The orthogonal groups $\mathrm{O}(N)$ and $\mathrm{SO}(N)$ are defined as
\[
    \mathrm{O}(N) := \{O \in GL_N(\R) ~:~ O^\top O = I\} \quad \text{and} \quad \mathrm{SO}(N) := \{O \in SL_N(\R) ~:~ O^\top O = I\}.
\]
These are real compact Lie groups of the same dimension $N(N-1)/2$.
Further, $\mathrm{SO}(N)$ is the identity component of $\mathrm{O}(N)$, and so $\mathrm{SO}(N)$ is connected while $\mathrm{O}(N)$ is not.
This in particular means that both groups have the same Lie algebra, given as
\[
    \mathfrak{o}(N) = \mathfrak{so}(N) = \{ X \in \mathfrak{sl}_N(\R) ~:~ X \text{ is skew-symmetric}\}.
\]
This Lie algebra has trivial center, and is in fact simple, and so the Killing form is the unique invariant symmetric bilinear form up to scalar.
So it must be a multiple of the Frobenius inner product.
The fact that $\mathrm{SO}(N)$ is compact implies $B$ is negative semidefinite, but this can also be seen from the fact that $\mathfrak{so}(N)$ consists solely of real skew-symmetric matrices.

Beyond these properties, the orthogonal groups break into two cases corresponding to even and odd values of $N$.
This distinction is standard in Lie theory, and we will see it explicitly in differences in the counting oracles for the two cases.

\subsubsection{Even case}

In this section, we consider the Lie group $\mathrm{SO}(2n)$.
We choose the standard maximal torus $T$ given by the block-diagonal matrices (with $2 \times 2$ blocks) of $\mathrm{SO}(2n)$, which corresponds to the Cartan subalgebra $\mathfrak{h}$ consisting of block-diagonal matrices (with $2 \times 2$ blocks) of $\mathfrak{so}(2n)$.
An orthonormal basis for $\mathfrak{h}$ and $\mathfrak{g}$ is then given by
\[
    H_j := \frac{E_{2j-1,2j} - E_{2j,2j-1}}{\sqrt{2}} \quad \text{for } j \in [n] \qquad \text{and} \qquad X_{j,k} := \frac{E_{j,k} - E_{k,j}}{\sqrt{2}} \quad \text{for all other } j < k.
\]
We now consider matrices $A,B \in \mathfrak{h}$ given by
\[
    A = \frac{1}{\sqrt{2}}\sum_{j=1}^n a_j H_j \qquad \text{and} \qquad B = \frac{1}{\sqrt{2}}\sum_{j=1}^n b_j H_j
\]
for real $a_j,b_j$.
Equations 24 and 27 of \cite{mcswiggen2018harish} then give the formulas
\[
    \int_{\mathrm{SO}(2n)} e^{-\Tr(AOBO^\top)} dO = \left(\prod_{p=1}^{n-1} (2p)!\right) \frac{\det[\cosh(2a_jb_k)]_{j,k=1}^n + \det[\sinh(2a_jb_k)]_{j,k=1}^n}{\prod_{j<k} (a_k^2 - a_j^2)(b_k^2 - b_j^2)}
\]
and
\[
    \int_{\mathrm{O}(2n)} e^{-\Tr(AOBO^\top)} dO = \left(\prod_{p=1}^{n-1} (2p)!\right) \frac{\det[\cosh(2a_jb_k)]_{j,k=1}^n}{\prod_{j<k} (a_k^2 - a_j^2)(b_k^2 - b_j^2)}
\]
for distinct $a_j,b_j$.

\subsubsection{Odd case}

In this section, we consider the Lie group $\mathrm{SO}(2n+1)$.
We choose the standard maximal torus $T$ given by the block-diagonal matrices of $\mathrm{SO}(2n+1)$ which have $n$ blocks of size $2 \times 2$, followed by a single $1 \times 1$ block.
This corresponds to the Cartan subalgebra $\mathfrak{h}$ consisting of block-diagonal matrices of $\mathfrak{so}(2n+1)$ which have the same block pattern as in $T$.
Note that this means that every matrix in $\mathfrak{h}$ has the property that all entries of the last row and column are 0.
An orthonormal basis for $\mathfrak{h}$ and $\mathfrak{g}$ is then given by
\[
    H_j := \frac{E_{2j-1,2j} - E_{2j,2j-1}}{\sqrt{2}} \quad \text{for } j \in [n] \qquad \text{and} \qquad X_{j,k} := \frac{E_{j,k} - E_{k,j}}{\sqrt{2}} \quad \text{for all other } j < k.
\]
We now consider matrices $A,B \in \mathfrak{h}$ given by
\[
    A = \frac{1}{\sqrt{2}} \sum_{j=1}^n a_j H_j \qquad \text{and} \qquad B = \frac{1}{\sqrt{2}} \sum_{j=1}^n b_j H_j
\]
for real $a_j,b_j$.
Equations 29 and 30 of \cite{mcswiggen2018harish} then give the formula
\[
\begin{split}
    \int_{\mathrm{O}(2n+1)} e^{-\Tr(AOBO^\top)} dO &= \int_{\mathrm{SO}(2n+1)} e^{-\Tr(AOBO^\top)} dO \\
        &= \left(\prod_{p=1}^{n-1} (2p+1)!\right) \frac{\det[\sinh(2a_jb_k)]_{j,k=1}^n}{\prod_{j=1}^n a_jb_j \prod_{j<k} (a_k^2 - a_j^2)(b_k^2 - b_j^2)}
\end{split}
\]
for distinct $a_j,b_j$ all non-zero.

\subsection{Symplectic groups}

The unitary symplectic group $\mathrm{USp}(n)$ is defined as
\[
\mathrm{USp}(n) := \mathrm{U}(2n) \cap \mathrm{Sp}(2n,\C) := \left\{\begin{bmatrix} A & -\overline{B} \\ B & \overline{A} \end{bmatrix} \in \mathrm{U}(2n)\right\},
\]
where $\overline{A}$ denotes the entrywise complex conjugation of $A$.
This is a real compact connected Lie group of dimension $n(2n+1)$.
The associated Lie algebra is given as
\[
    \mathfrak{usp}(n) = \left\{\begin{bmatrix}
        A & -\overline{B} \\ B & \overline{A}
    \end{bmatrix} ~\bigg|~ B=B^\top \text{ and } A=-A^*\right\}.
\]
This Lie algebra has trivial center, and is in fact simple, and so the Killing form is the unique invariant symmetric bilinear form up to scalar.
So it must be the Frobenius inner product.
The fact that $\mathrm{USp}(n)$ is compact implies $B$ is negative semidefinite, but this can also be seen from the fact that $\mathfrak{usp}(n)$ consists solely of skew-Hermitian matrices.

We choose the standard maximal torus $T$ given by the diagonal matrices of $\mathrm{USp}(n)$, which corresponds to the Cartan subalgebra $\mathfrak{h}$ consisting of diagonal matrices of $\mathfrak{usp}(n)$.
An orthonormal basis for $\mathfrak{h}$ is then given by
\[
    H_j := \frac{i(E_{j,j} - E_{n+j,n+j})}{\sqrt{2}} \quad \text{for } j \in [n].
\]
We omit the rest of the orthonormal basis of $\mathfrak{g}$, because it is more complicated and we don't use it further in the computations.
That said, the remaining basis elements are straightforward to derive and can take on a form similar to the other groups if desired.
We now consider matrices $A,B \in \mathfrak{h}$ given by
\[
    A = \sqrt{2} \sum_{j=1}^n a_j H_j \qquad \text{and} \qquad B = \sqrt{2} \sum_{j=1}^n b_j H_j
\]
for real $a_j,b_j$.
(Note that multiplication by $\sqrt{2}$ here is intentional, as compared to the division by $\sqrt{2}$ in the orthogonal case.)
Equation 32 of \cite{mcswiggen2018harish} then gives the formula
\[
    \int_{\mathrm{USp}(n)} e^{-\Tr(ASBS^\top)} dS = \left(\prod_{p=1}^{n-1} (2p+1)!\right) \frac{\det[\sinh(2a_jb_k)]_{j,k=1}^n}{\prod_{j=1}^n a_jb_j \prod_{j<k} (a_k^2 - a_j^2)(b_k^2 - b_j^2)}
\]
for distinct $a_j,b_j$ all non-zero.
Note that this formula is exactly the same as in the odd orthogonal case, though the values of $a_j$ and $b_j$ take on different meanings.

\subsection{Other groups} \label{sec:app_other_groups}

Our main results apply to the case of compact connected groups in general, and so they apply to any such group for which we can compute the above integral.
In particular if $G$ is any product of the above groups, then the integral over $G$ boils down to a product of integrals and therefore a product of determinants.
Further, there exist other compact connected groups than those discussed above: the spin groups which are the covering groups of $\mathrm{SO}(n)$ and the compact forms of the exceptional Lie groups, for example.
Up to finite extensions and finite covers, this in fact exhausts the list of all compact connected (real) Lie groups.
We refer the reader to \cite{mcswiggen2018harish} for further discussion on how one might go about deriving determinantal formulas for integrals over such groups.

Further, in terms of Lie algebras and Weyl groups, the only ones we do not consider here are those associated to the exceptional Lie groups.
There are only finitely many of these, and so they cannot substantially contribute to any asymptotic complexity statement.
These means that in principle we could remove the assumption of a strong counting oracle for our main algorithm in Theorem \ref{thm:lie_algo}, because have given a strong counting oracle for all but finitely many base cases.
This is somewhat outside of the spirit of our determinantal formula counting oracles, and so we omit further exposition.

\section*{Acknowledgments}
The authors would like to thank Colin McSwiggen for useful comments. This research was partially supported by NSF CCF-1908347 grant. 

\bibliographystyle{plain}
\bibliography{references}

\end{document}